\documentclass[11pt, twocolumn]{article}
\usepackage{geometry}
\usepackage[T1]{fontenc}
\usepackage[utf8]{inputenc}
\usepackage{authblk}

\usepackage{amsmath}
\usepackage{amsfonts}
\usepackage{amsthm}
\usepackage{graphicx}

\usepackage{algorithm}
\usepackage{algpseudocode}
\usepackage{cleveref}
\usepackage{xcolor}

\newcommand{\R}{\mathbb{R}}

\newcommand{\N}{\mathbb{N}}
\newcommand{\calS}{\mathcal{S}}

\newcommand{\eps}{\varepsilon}

\DeclareMathOperator*{\argmin}{argmin}

\DeclareMathOperator{\opspan}{span}
\newcommand{\norm}[1]{{\left\| #1\right\|}}
\newcommand{\normsq}[1]{\norm{#1}^2}

\newtheorem{definition}{Definition}
\newtheorem{theorem}{Theorem}
\newtheorem{lemma}{Lemma}
\newtheorem{assumption}{Assumption}

\author[1]{Erin George$^\S$}
\author[1]{Yotam Yaniv$^\S$}
\author[1]{Deanna Needell}
\affil[1]{University of California, Los Angeles, Department of Mathematics}
\affil[ ]{\texttt {\{egeo, yotamya, deanna\}@math.ucla.edu}}

\date{\vspace{-5ex}}
\begin{document}

\title{Multi-Randomized Kaczmarz for Latent Class Regression}

\maketitle

\def\thefootnote{\S}\footnotetext{These authors contributed equally to this work}

\begin{abstract}
Linear regression is effective at identifying interpretable trends in a data set, but averages out potentially different effects on subgroups within data.  We propose an iterative algorithm based on the randomized Kaczmarz (RK) method to automatically identify subgroups in data and perform linear regression on these groups simultaneously.  We prove almost sure convergence for this method, as well as linear convergence in expectation under certain conditions.  The result is an interpretable collection of different weight vectors for the regressor variables that capture the different trends within data.  Furthermore, we experimentally validate our convergence results by demonstrating the method can successfully identify two trends within simulated data.
\end{abstract}

\section{Introduction}
Often, one needs to perform regression tasks on extremely large-scale data. Methods such as the randomized Kaczmarz method (RK) \cite{K37:Angena,strohmer2009randomized} have gained recent attention for their ability to solve such systems with needing to only access a single row at a time rather than the full system in memory. However, in many settings, two or more population subgroups may be present in the data requiring multiple regressors. Often times, computing a single regressor will result in a minority group having far worse predictive power than the majority. Additionally, the minority group is not known a priori requiring that we both discover and regress on these subgroups on the fly.
Here, we present a variant of RK that addresses this problem via multiple regressors.

Formally, given multiple consistent systems of equations $M^{(i)} x^{(i)}_* = b^{(i)}$, $i \in \{0, 1, \dots, n\}$ we consider the combined matrices
\begin{equation*}
M = \begin{bmatrix}
M^{(0)} \\
M^{(1)} \\
\vdots \\
M^{(n)}
\end{bmatrix}~
b = \begin{bmatrix}
b^{(0)} \\
b^{(1)} \\
\vdots \\
b^{(n)}
\end{bmatrix} 
\end{equation*} with the goal of recovering $x^{(0)}_*, x^{(1)}_*, \dots, x^{(n)}_*$  where the rows of these matrices may be shuffled. 
Next we define the class of a set of rows and right hand side entries. 
\begin{definition}[Class]
    Given a regressor $x^{(i)}_*$ a set of rows resulting in matrix $M^{(j)}$ and right hand sides $b^{(j)}$ are in class $i$ if $M^{(j)}x^{(i)}_* = b^{(j)}$.
\end{definition}
We assume that the class of each row is not known beforehand. This task corresponds with uncovering multiple systems and their solutions. In the statistics literature, this problem can be framed as latent class linear regression where each class represents an overdetermined system of equations \cite{wedel1994review, magidson2004latent}. Classically, this problem can be solved by using an expectation-maximization (EM) algorithm to iteratively fit the regressor coefficients and then classify the rows~\cite{moon1996expectation, klusowski2019estimating}. The EM algorithm has been extensively studied in a statistics framework with convergence properties discussed in \cite{wu1983convergence} and \cite{dempster1977maximum}. More recently the EM framework has been used to learn class data representations in unsupervised machine learning using neural networks \cite{greff2017neural}.

We take a randomized numerical linear algebra approach to this problem by modifying the classical randomized Kaczmarz algorithm to this setting. This approach allows us to process very large data sets while only accessing single rows of our data set at a time. 

\section{Multi-Randomized Kaczmarz Method} 

We propose a novel iterative method motivated by the randomized Kaczmarz (RK) algorithm for simultaneously solving all $n+1$ systems, Algorithm~\ref{alg:mrk}. This approach is motivated by the assumption that the closer an iterate is to a hyperplane defined by a row of the combined system, the more likely that row belongs to the class of that iterate. Since Kaczmarz methods converge monotonically this is a reasonable assumption.

At each iteration, the multi-randomized Kaczmarz (MRK) method selects a hyperplane as in the standard RK algorithm. Then the Kaczmarz update for all iterates is computed. The update with the smallest magnitude is selected, denoted $s_k$, with the respective magnitude denoted as $c_{s_k}$. Given a swap probability $r$ we then update iterate $t_k$ chosen to be $s_k$ with probability $1-r$ and $t_k$ chosen from all iterates uniformly at random with total probability $r$. The selected iterate $t_k$ is updated by the magnitude $c_{s_k}$ in the direction the $t_k$-th iterate would have been updated given the standard Kaczmarz update.   

\begin{algorithm}
\caption{Multi-Randomized Kaczmarz (MRK) Algorithm}\label{alg:mrk}
\begin{algorithmic}
\State \textbf{Input:} System $M$, right hand side $b$, number of iterations $N$, initial iterates $x_0^{(0)}$, $x_0^{(1)}, \dots, x_0^{(n)}$, swap probability $r$, sampling distribution $\mathcal{D}$.
\For{$k$ from 0 to $N-1$}
\State Sample row $i_k \sim \mathcal{D}$
\State $c_{i,k} = \frac{M_{i_k}x_k^{(i)} - b_{i_k}}{||M_{i_k}||^2}$, $i = 0, 1, \dots, n$ 
\State $s_k = \argmin_{i \in  \{0, 1, \dots, n\}}(|c_{i,k}|)$ 
\State $t_k = \begin{cases}s_k &\text{ with probability $1-r$}\\ t&\text{ with probability $\tfrac{r}{n+1}$ for all $t\in\{0,\ldots,n\}$}\end{cases}$
\\\Comment The total probability that $t_k = s_k$ is $1-r+\tfrac{r}{n+1}$.
\State $x_{k+1}^{(t_k)} = x_{k}^{(t_k)} - |c_{s_k}|\operatorname{sgn}(c_{t_k}) M_{i_k}^T$ 
\State$x_{k+1}^{(j)} = x_{k}^{(j)}, j \neq t_k$
\EndFor
\end{algorithmic}
\end{algorithm}

We state two convergence results for this method, which we will prove in the following section. The first theorem, \Cref{thm:cond_exp_conv}, proves a linear convergence result for the MRK algorithm in expectation under certain conditions. The second theorem, \Cref{thm:almost_sure}, is an almost sure convergence result for the MRK algorithm. Other almost sure convergence results have been shown for Kaczmarz type algorithms~\cite{chen2012almost} under the assumption that measurements (rows of the matrix) are drawn from independent but not necessarily identical distributions.

To prove these theorems, we will make a uniqueness assumption on the problem.
\begin{assumption}
The solution to the set of systems is unique up to relabeling.  That is, suppose there are $x_i$, $i \in \{0,1,\ldots,n\}$ so that for each row in the combined system (indexed by $k$) there is $i_k$ where
\[M_k x_{i_k} = b_k.\]
Then there is a permutation $\sigma$ on $\{0,1,\ldots,n\}$ so that $x_*^{(i)} = x_{\sigma(i)}$ for all $i$.\label{asm:unique}
\end{assumption}
In particular, this means all systems in the problem are full rank, even if rows which consistently belong to two or more classes are removed.
\begin{theorem}[Conditional expected MRK Convergence] \label{thm:cond_exp_conv} Define
\[e_k = \sum_{i=0}^n \normsq{x_k^{(i)} - x^{(i)}_*}.\]
Let $r \geq 0$ be sufficiently small.  Choose $c \in (C_0, 1)$ and $\delta > 0$.  There exists $\varepsilon > 0$ so that if $e_k < \varepsilon$ then
\begin{equation*}
    \mathbb{E}(e_{k+b}| A_\delta) \leq c^b e_k
\end{equation*}
where $A_\delta$ is an event that happens with probability at least $1-\delta$.  The constant $C_0 < 1$ depends on $M$, $b$, $n$, and $r$.
\end{theorem}

This theorem shows the convergence will be linear in expected squared error after a certain point.  Limiting the initial squared error before convergence allows us to identify which solution each iterate is converging towards.  The failure probability reflects the possibility that the iterates may still converge towards a different labeling of the solutions and iterates.  In the case where the initial squared error is too large or the failure probability is triggered, we will still see convergence, as shown in the next theorem.

\begin{theorem}[Almost sure MRK Convergence]
\label{thm:almost_sure}
There is $r' \in (0,1)$ so that if $r \in (0,r')$, each iterate of the algorithm converges almost surely to a different solution of the subsystems.
\end{theorem}
The convergence rate given by the proof of this theorem is slow.  In practice, we find the convergence rate quickly achieves the linear rate given in the previous theorem.

\section{Proofs}

\subsection{Conditional Convergence in Expectation}

\begin{proof}[Proof of Theorem~\ref{thm:cond_exp_conv}]

At the $k$-th iteration, we select a row from a system.  Suppose we select a row $\ell_k$ from system $i$.  There are three possibilities for how we update.
\begin{enumerate}
    \item[(a)] We update $x^{(i)}_k$ fully, setting
    \[\normsq{x^{(i)}_{k+1} - x^{(i)}_*} = C^{(i)}_k \normsq{x^{(i)}_k - x^{(i)}_*}\]
    for some random variable $C^{(i)}_k$ taking value in the range $[0,1]$.  The expectation of $C^{(i)}_k$ is just the Kaczmarz constant for the subset of rows which we are allowed to make a full and correct update with.
    \item[(b)] We update $x^{(i)}_k$ partially.  We bound the error here as
    \[\norm{x^{(i)}_{k+1} - x^{(i)}_*} \leq \norm{x^{(i)}_k - x^{(i)}_*}.\]
    \item[(c)] We update $x^{(j)}_k$ for some $j \neq i$.  Regardless of how this happens, we always update by a magnitude bounded above in norm by the correct update:
    \[\frac{\left|M_{\ell_k} x^{(i)}_k - b_{\ell_k}\right|}{\norm{M_{\ell_k}}} \leq \norm{x^{(i)}_k - x_*^{(i)}}.\]
    Therefore the new error satisfies
    \[\norm{x^{(j)}_{k+1} - x^{(j)}_*} \leq \norm{x^{(j)}_k - x^{(j)}_*} + \norm{x^{(i)}_k - x^{(i)}_*}\]
    and by Cauchy-Schwarz and Young's inequality
    \[\normsq{x^{(j)}_{k+1} - x^{(j)}_*} \leq 2\normsq{x^{(j)}_k - x^{(j)}_*} + 2\normsq{x^{(i)}_k - x^{(i)}_*}.\]
\end{enumerate}
There are two ways for us to land in case (c).  Either we trigger our swap probability and select iterate $j$, or we do not trigger our swap probability but we selected iterate $j$ anyway.  The second happens only when
\[\frac{\left|M_{\ell_k} x^{(j)}_k - b_{\ell_k}\right|}{\norm{M_{\ell_k}}} \leq \frac{\left|M_{\ell_k} x^{(i)}_k - b_{\ell_k}\right|}{\norm{M_{\ell_k}}}\]
We can bound the left side below by
\[\frac{|M_{\ell_k}(x_*^{(i)}-x_*^{(j)})|}{\norm{M_{\ell_k}}} - \norm{x_k^{(j)}-x_*^{(j)}}\]
and the right hand side above by
\[\norm{x^{(i)}_k - x_*^{(i)}}.\]
So this can only happen when
\begin{align*}
\frac{|M_{\ell_k}(x_*^{(i)}-x_*^{(j)})|}{\norm{M_{\ell_k}}} &\leq \norm{x^{(i)}_k - x_*^{(i)}} + \norm{x^{(j)}_k - x_*^{(j)}}\\
&\leq \sum_{a=0}^n \norm{x^{(a)}_k - x_*^{(a)}}\\ &\leq \sqrt{(n+1) \cdot e_k}.
\end{align*}
We only need to consider the case when $M_{\ell_k}(x_*^{(i)}-x_*^{(j)}) \neq 0$, as otherwise we could consider this row $\ell_k$ as coming from the $j$-th system anyway.  Therefore, the probability of this happening goes to 0 as $e_k$ goes to 0.  Suppose $\eps$ is small enough so that whenever $e_k < \eps$ the probability this mistake happens for any pair is less than $q$.

We will also assume that $\eps$ is small enough so that, assuming we do not trigger our swap probability, there is a full rank set of rows for each system so that whenever $e_k < \eps$ all these rows will make a correct update and that cannot be added to another system consistently.  This is a consequence of~\Cref{asm:unique}.  Then, for each system, the condition number for the set of rows that can make a correct update is bounded above, and the RK constant is bounded above by some value strictly less than 1.  Let $c < 1$ bound above the RK constant for each system.

Now, whenever $e_k < \eps$, we can bound
\[\begin{pmatrix}\normsq{x_{k+1}^{(0)}-x_*^{(0)}} \\\vdots\\ \normsq{x_{k+1}^{(n)}-x_*^{(n)}} \end{pmatrix} \leq \mathbf{A}\begin{pmatrix}\normsq{x_{k}^{(0)}-x_*^{(0)}} \\\vdots\\ \normsq{x_{k}^{(n)}-x_*^{(n)}} \end{pmatrix}\]
where $\leq$ is interpreted component-wise and
\begin{align*}
 \mathbf{A}_{ii} &= 1 + \frac{m_j}{m}(c-1)(1-q-\tfrac{nr}{n+1}) + \frac{m-m_j}{m}(q+\tfrac{r}{n+1}) \\
\mathbf{A}_{ij} &= 2\frac{m_j}{m}(q+\tfrac{r}{n+1}) \; \text{if $i\neq j$.}   
\end{align*}

Here $m_j$ is the number of rows in the $j$-th system and $m = \sum_j m_j$.

By induction,
\[\begin{pmatrix}\normsq{x_{k+b}^{(0)}-x_*^{(0)}} \\\vdots\\ \normsq{x_{k+b}^{(n)}-x_*^{(n)}} \end{pmatrix} \leq \mathbf{A}^b\begin{pmatrix}\normsq{x_{k}^{(0)}-x_*^{(0)}} \\\vdots\\ \normsq{x_{k}^{(n)}-x_*^{(n)}} \end{pmatrix}\]
for all $b \in \N$ provided that $e_{k+a} < \eps$ for all $a \in \{0,\ldots,b-1\}$.  We wish to show the $\ell_1$ operator norm of $\mathbf{A}$ is less than 1.  This happens when
\begin{align*}
    &\frac{m_j}{m}(c-1)\left(1-q-\frac{nr}{n+1}\right) \\&+ \left(q+\frac{r}{n+1}\right)\left(\frac{m-m_j}{m}+ 2n\frac{m_j}{m}\right)
\end{align*}
is negative.  This occurs when $q + \frac{nr}{n+1}$ is small enough.  So then, for $r$ sufficiently small, we can choose $\eps$ to make $\norm{\mathbf{A}}_{\ell^1\to\ell^1} = d < 1$.

Suppose $e_k < \delta < \eps$.  By our previous bound, $e_{k+b} < d^b \delta$, conditioned on the intermediate values $e_{k+a} < \eps$.  By Markov's inequality, the probability that $e_{k+a} \geq \eps$ is at most $\frac{1}{\eps}d^a \delta$. The total probability of this happening is at most $\frac{\delta}{\eps}\frac{d}{1-d}$.  Therefore our total error remains bounded above by $\eps$ with probability at least $1 - \frac{\delta}{\eps}\frac{d}{1-d}$, and in this case we have convergence in expectation.

\end{proof}

\subsection{Convergence with full probability}

An outline of the proof for Theorem~\ref{thm:almost_sure}:
\begin{enumerate}
    \item We use Theorem~\ref{thm:cond_exp_conv} to define ``convergence basins'': regions where, if the iterates fall into, there is some positive probability that they never escape and converge in expectation.
    \item We show $\norm{x_k^{(i)}-x_*^{(i)}}$ is bounded by some constant independent of $i$ and $k$.
    \item We show we can bound the probability of falling into a basin eventually below by some positive number.
\end{enumerate}

We will begin by proving the following lemma, which will be used in the second part of the outline above.
\begin{lemma}
Let $R$ be a sequence of rows from the problem.  The sequence of Kaczmarz updates corresponding to $R$ defines an affine transformation $v \mapsto T_R v + v_R$.  There are constants $c_r \in (0,1), B_r \in \R_+$ for $r \in \{1,\ldots,d\}$ so that $\norm{v_R} \leq B_{\dim\opspan R}$ and $\norm{T_R}_R \leq c_{\dim\opspan R}$, where $\norm{\cdot}_R$ is the $\ell^2$ operator norm when the operator is restricted to $\opspan R$.\label{lem:bdd_op_norms}
\end{lemma}

\begin{proof}[Proof of Lemma~\ref{lem:bdd_op_norms}]

We proceed by induction on $r$.  The Kaczmarz update for the $\ell$-th row is
\[K_\ell : v \mapsto \left(I - \frac{1}{\normsq{M_\ell}}M_\ell^T M_\ell\right)v + \frac{b_\ell}{\normsq{M_\ell}}M_\ell^T\]

If $r=1$, then $T_R$ is the zero operator restricted to $\opspan R$.  We can take $c_1 = 0$ and $B_1 = \max_\ell\frac{|b_\ell|}{\norm{M_\ell}}$.

Now, assume the lemma is true for all $r < r'$.  Let $R = (M_{\ell_1}, \ldots, M_{\ell_k})$ be a sequence of rows where $\dim\opspan R = r'$.  We can group
\begin{align*}K_{\ell_k} \circ \cdots \circ K_{\ell_1} =&
(K_{\ell_k} \circ \cdots \circ K_{\ell_{a_N}}) \\
&\circ
(K_{\ell_{a_N-1}}\circ\cdots\circ K_{\ell_{a_{N-1}}})\\
&~\vdots\\
&\circ(K_{\ell_{a_2-1}}\circ\cdots\circ K_{\ell_{a_1}})\\
&\circ (K_{\ell_{a_1-1}}\circ\cdots\circ K_{\ell_{a_0}})
\end{align*}
so that $a_0 = 1$ and $a_i$ for $i\in\{0,\ldots,N\}$ is a strictly increasing sequence where the following are true:
\begin{align*}
    \dim\opspan\{M_{\ell_{a_{i-1}}}, \ldots, M_{\ell_{a_{i}-1}}\} &= r'&\forall i \in \{1,\ldots,N\} \\
    \dim\opspan\{M_{\ell_{a_{i-1}}}, \ldots, M_{\ell_{a_{i}-2}}\} &= r'-1&\forall i \in \{1,\ldots,N\}\\
    \dim\opspan\{M_{\ell_{a_n}},\ldots,M_{\ell_k}\} &< r'.
\end{align*}
Consider a grouping $(K_{\ell_{a_i} -1} \circ \cdots \circ K_{\ell_{a_{i-1}}})$, the linear part of the transformation is $A = \left(I - \frac{1}{\normsq{M_{\ell_{a_i}}}}M_{\ell_{a_i}}^T M_{\ell_{a_i}}\right)$ composed with an operator $T$ that sends $\calS = \opspan\{M_{\ell_{a_{i-1}}}, \ldots, M_{\ell_{a_{i}-2}}\}$ to itself and has operator norm less than $c_{r'-1}$ on this space.  Consider a unit vector $v\in\opspan R$.  We can decompose $v = v_1 + v_2$ with $v_1 \in \calS$ and $v_2 \in \calS^\perp \cap \opspan R$.  This allows us to bound \[\norm{ATv} \leq \norm{Tv} \leq \sqrt{c_{r'-1}^2\normsq{v_1} + \normsq{v_2}}\]
using that the operator norm of $A$ is at most 1 and that $T$ is the identity on $\calS^\perp$.
Another bound is
\[\norm{ATv} \leq \norm{ATv_1} + \norm{Av_2} \leq \norm{v_1} + \left(1 - \frac{|M_{\ell_{a_i}}v_2|}{\norm{M_{\ell_{a_i}}}}\right)\]
obtained with the triangle inequality and again $T$ being the identity on $\calS^\perp$.

These bounds combine to give a bound for the $\ell^2$ operator norm of $AT$ on $\opspan R$ that depends only on $\calS$ and $\ell_{a_i}$.  There are finitely many possible choices for $\calS$ and $\ell_{a_i}$, so there is some bound $c' < 1$ on the operator norm of $AT$ independent of what $\calS$ and $\ell_{a_i}$ are.  Next we turn to the affine part.  By the induction hypothesis, this is a vector with norm at most $B' = B_{r'-1} + \max_\ell\frac{|b_\ell|}{\norm{M_\ell}}$.

Next we turn to the last grouping, which is not of this form.  We will not analyze the linear part, and note the affine part $v'$ is bounded above in norm by $B'' = \max_{r < r'} B_r$.

Since all linear operators here have operator norm at most 1, a final bound for the operator norm of $T_R$ restricted to $\opspan R$ is $c'$, which we can take to be $c_{r'}$.  We bound
\[\norm{v_R} \leq B'' + \left[\sum_{i=1}^\ell c_{r'}^{i-1} B'\right] \leq B'' + \frac{1}{1-c_{r'}} B'.\]
So we can let $B_{r'}$ be this bound.
\end{proof}

\begin{proof}[Proof of Theorem~\ref{thm:almost_sure}]
We first bound the norm of the iterates above by some constant $D$.

Consider the evolution of a single iterate $x_k^{(i)}$ after finitely many steps.  At the last update for $x_k^{(i)}$, we perform a Kaczmarz update with respect to some system, and move the iterate towards, but not past, the update.  There is a line segment of possible choices for the next update once we have selected the row.  The potential next iterate with largest norm is one of the end points of the line segment, corresponding to either doing a full update or no update at all.  So we can either remove this last update or replace it with a full update to yield a final iterate with norm at least as large.  We repeat with each update in reverse order, noting the image of a line segment after a series of affine transformations is still a line segment, choosing the one that will yield the largest norm at the end.  Hence, to bound the norm of the iterates, we only need to consider sequences where we only ever make full Kaczmarz updates.

Using Lemma~\ref{lem:bdd_op_norms}, we see that if the initial norm of the iterates are bounded above by $A$, after a sequence of rows $R$ the norm is at most $c_{\dim\opspan R} A + B_{\dim\opspan R}$, which is bounded above by $D = A + \max_{r \in \{1,\ldots,d\}} B_r$.

As given by Theorem~\ref{thm:cond_exp_conv}, let $\eps$ be such that if $e_k \leq \eps$, we converge with some positive probability $t$.

Let $C$ be the set of all possible iterates such that $\tfrac{\eps}{n+1} \leq \normsq{x_k^{(i)}-x_*^{(i)}} \leq D^2$ for all $i \in \{0,\ldots,n\}$ with $D$ given above.  This set is compact.  If our iterates do not lie in $C$, then already $e_k < \eps$, so we only need to look at what happens if our iterates lie in $C$.

Define
\[g(x^{(0)},\ldots,x^{(n)}) = \max_{\ell \in \{1,\ldots,m\}}\min_{i \in \{0,\ldots,n\}} \frac{
|M_\ell x^{(i)}-b_\ell|}{\norm{M_\ell}}.\]
This function is the norm of the largest possible iterates in our algorithm if the iterates currently take the values $x^{(0)},\ldots,x^{(n)}$.  This is a continuous function on $\R^{d(n+1)}$, so it achieves a minimum $c$ on $C$.  This minimum $c$ is positive, as by our assumption $g$ cannot be zero anywhere on $C$.

Now, for all value of iterates, we have probability at least $\frac{1}{m}$ of choosing a row where we will make an update with norm at least $c$.  The probability of updating the correct system with a chosen row is at least $\frac{r}{n+1}$.  The squared error of the corresponding iterate decreases by \emph{at least} the norm squared of the update, which is at least $c^2$, because the resulting triangle between the previous iterate, next iterate, and solution is obtuse.  We can keep doing this as long as our iterates remain in $C$. Therefore, if we make no more than
\[A = (n+1)\left\lceil\frac{N-\tfrac{\eps}{n+1}}{c^2}\right\rceil\]
of these updates, we have $e_k < \eps$.  The probability of this happening is
\[{\left(\frac{r}{m(n+1)}\right)}^B,\]
where $m$ is the number of rows of $M$.  This is a fixed positive value independent of the iterate.
\end{proof}

\section{Experimental Results}
We test the MRK method on synthetic and real world data to verify the merits of the method. First we construct a problem in two-dimensional space and visualize how our iterates move in space in \Cref{fig:2rk_spatial_plot}. From \Cref{fig:2rk_spatial_plot} we see how the iterates converge to solutions, moving closer with each projection. The problem is defined by two $10 \times 2$ systems where $M_1 \in \R^{10 \times 2}$ with entries drawn i.i.d. from $\mathcal{N}(0.8,0.3)$ and $M_2 \in \R^{10 \times 2}$ with entries drawn i.i.d. from $\mathcal{N}(-0.8,0.3)$. 

\begin{figure}
    \centering
    \includegraphics[width= .95\linewidth]{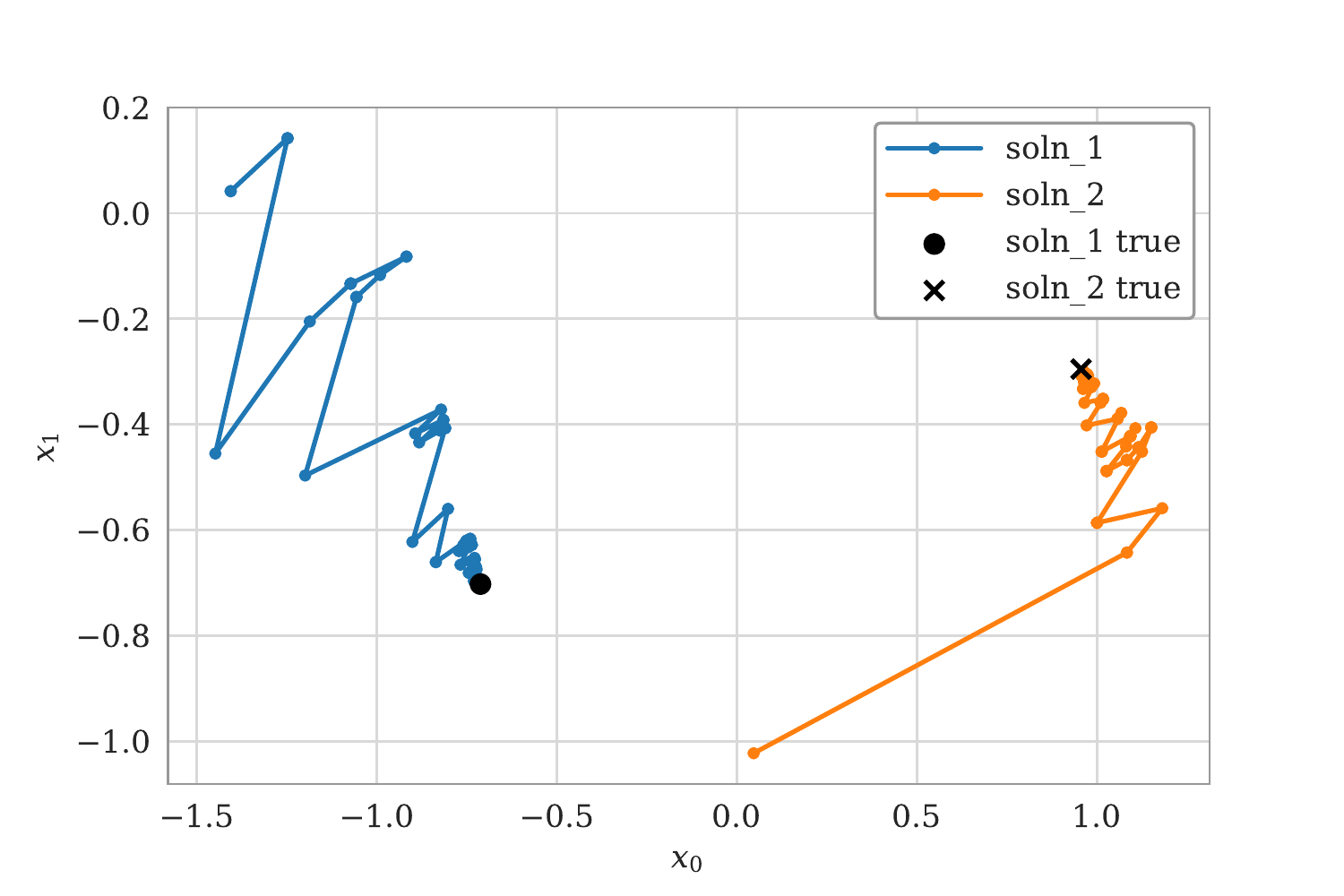}
    \caption{Here we plot the evolution of two iterates in the two-dimensional plane. Our system is defined by two matrices $M_1 \in \R^{10 \times 2}$ with entries drawn i.i.d. from $\mathcal{N}(0.8,0.3)$ and $M_2 \in \R^{10 \times 2}$ with entries drawn i.i.d. from $\mathcal{N}(-0.8,0.3)$. Each initial iterate $x_0^0, x_0^1 \sim \mathcal{N}(0,1)$. We let our swap probability $r=0$ and sample rows uniformly at random.}
    \label{fig:2rk_spatial_plot}
\end{figure}

Next, in \Cref{fig:2rk_convergence} we use the MRK method on a large synthetic data set. We plot the log norm squared error per iteration of the two iterates for the two class system. The system is defined by two matrices $M_1 \in \R^{1000 \times 10},M_2 \in \R^{1000 \times 10}$ where the matrices have entries drawn i.i.d. from $\mathcal{N}(0,1)$. Each initial iterate $x_0^0, x_0^1 \sim \mathcal{N}(0,1)$ with zero swap probability. We plot the median and shade the interquartile range in \Cref{fig:2rk_convergence}. We observe that both iterates converge to machine precision and the method succeeds at solving both systems simultaneously.
\begin{figure}
    \centering
    \includegraphics[width=0.9\linewidth]{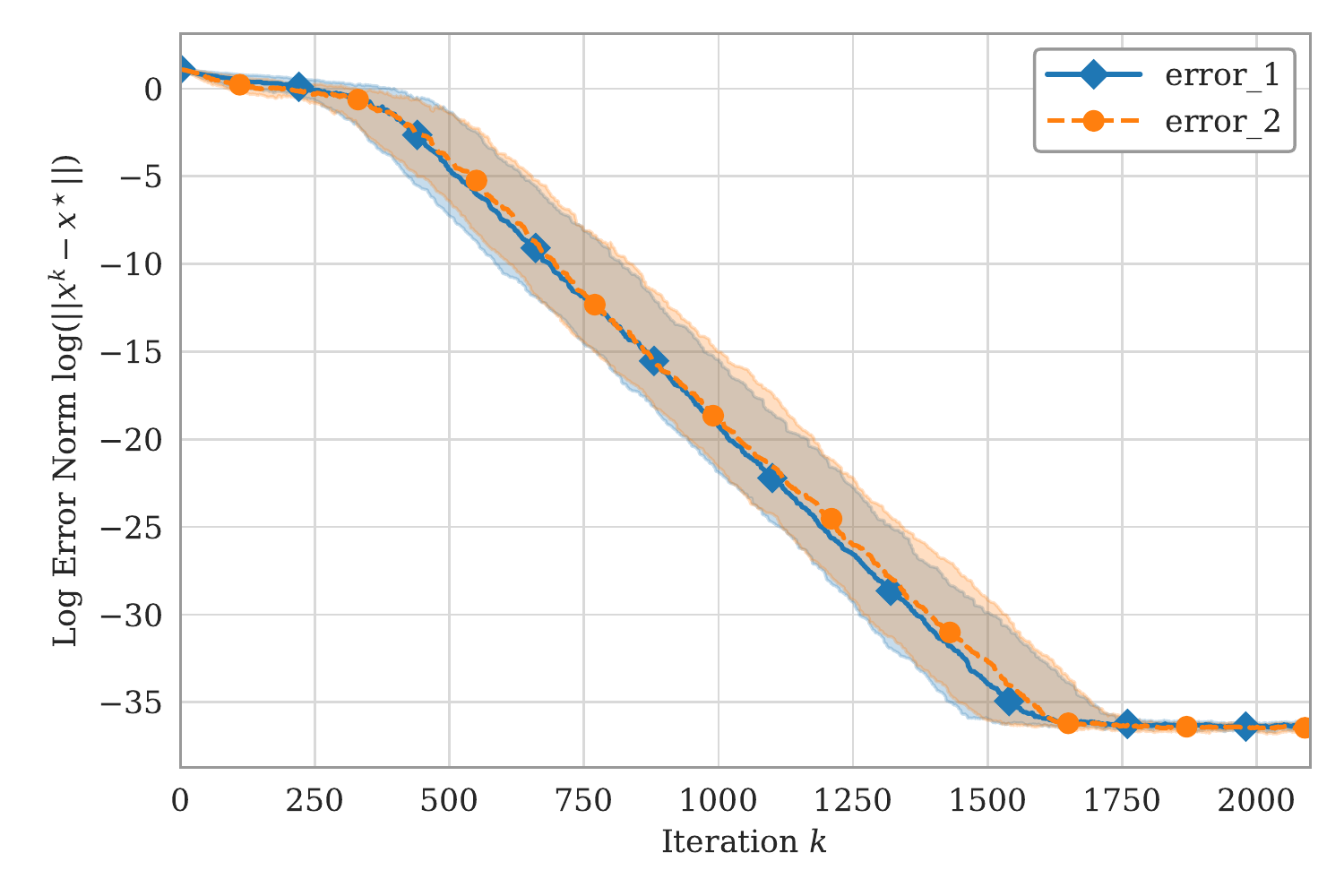}
    \caption{System defined by two matrices $M_1 \in \R^{1000 \times 10},M_2 \in \R^{1000 \times 10}$ where the matrices have entries distributed $M_1, M_2 \sim \mathcal{N}(0,1)$. Each initial iterate $x_0^0, x_0^1 \sim \mathcal{N}(0,1)$. We let our swap probability $r=0$, sample rows uniformly at random and plot the median and interquartile range over the 100 trials.
    }
    \label{fig:2rk_convergence}
\end{figure}

Finally, in \Cref{fig:2rk_convergence_wisc} we define a two problem system for the MRK method using real world data \cite{wolberg1995breast}. We construct two systems defined by matrices $M_1 \in \R^{300 \times 10},M_2 \in \R^{399 \times 10}$ submatrices of the Wisconsin breast cancer data set. We start two initial iterates with standard normal entries and let our swap probability be zero. We observe that in \Cref{fig:2rk_convergence_wisc} both iterates converge to their respective solutions. We plot the median and interquartile range for the iterates' log error norm per iteration over 100 trials and observe that the method is able to solve this two system problem. 

\begin{figure}
    \centering
    \includegraphics[width=0.9\linewidth]{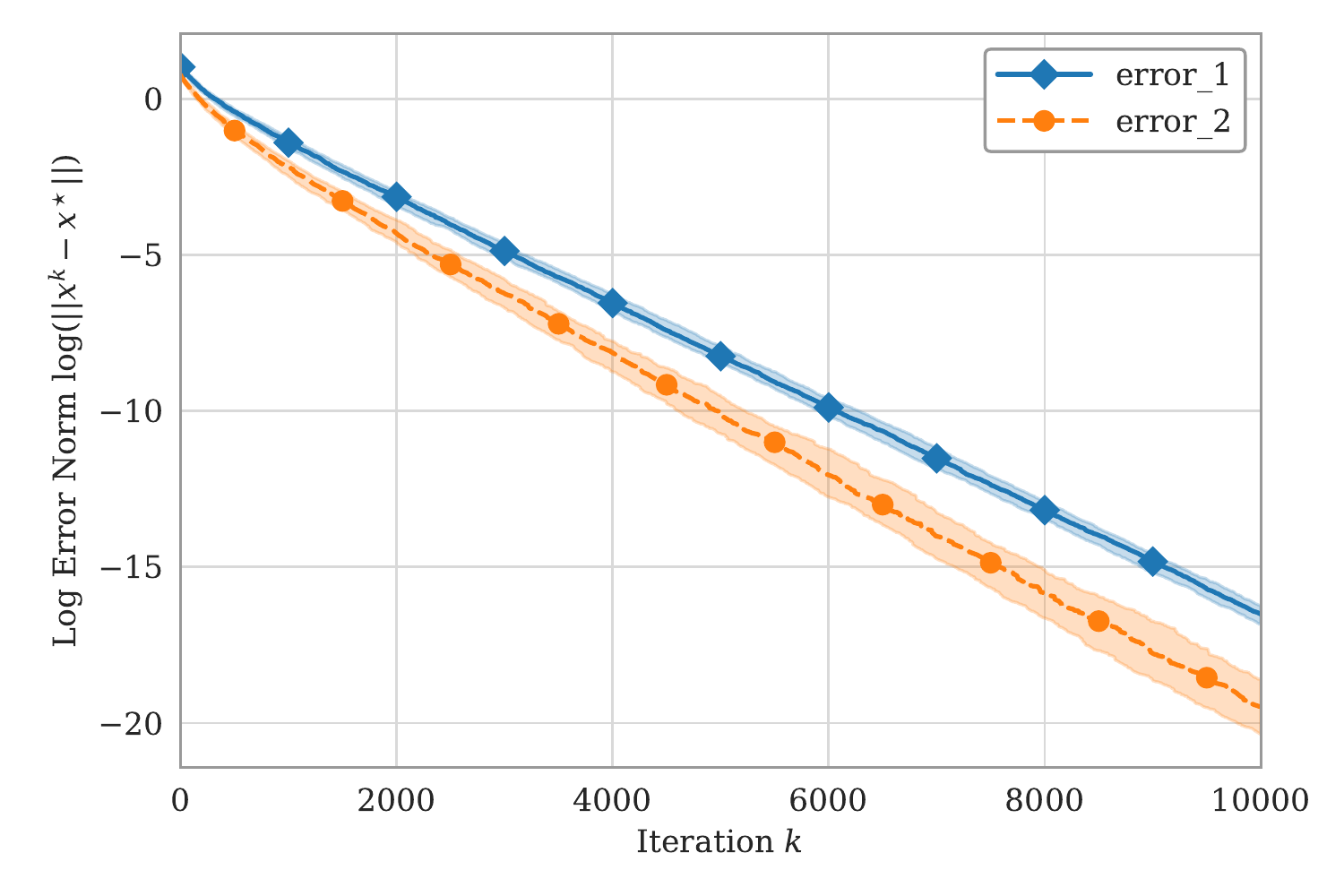}
    \caption{System defined by matrices $M_1 \in \R^{300 \times 10},M_2 \in \R^{399 \times 10}$ submatrices of the Wisconsin breast cancer data set. Each initial iterate $x_0^0, x_0^1 \sim \mathcal{N}(0,1)$. We let our swap probability $r=0$, sample rows uniformly at random and plot the median and interquartile range over the 100 trials.
    }
    \label{fig:2rk_convergence_wisc}
\end{figure}

\section{Conclusion and Future Work}
In this paper we introduce the novel multi-randomized Kaczmarz algorithm, \Cref{alg:mrk}, to solve the consistent latent class regression problem. We prove linear convergence for the algorithm in expectation with high probability under some constraints in \Cref{thm:cond_exp_conv} and almost surely in \Cref{thm:almost_sure}. Additionally, we observe promising results when applying the algorithm to test data sets. We plan on extending this work to inconsistent and noisy systems by leveraging the extended Kaczmarz method \cite{zouzias2013randomized}, which converges to the least squares solution \cite{needell2010randomized}. Additionally, we would like to explore using Kaczmarz variants such as max distance \cite{nutini2016convergence}, sampling Kaczmarz-Motzkin \cite{de2017sampling} and selectable set \cite{yaniv2022selectable} methods in this setting. Finally, we are interested in adaptively marking and assigning which rows belong to which system in real time based on the iterate projection values.

\bibliographystyle{IEEEtran}
\bibliography{IEEEabrv,ref}

\end{document}